%% file: DarmonPointsHigherConductor.tex
\documentclass[oneside]{amsart}
  
\usepackage{amssymb}

\usepackage{graphicx}
\usepackage{algorithm,algorithmic}

\usepackage[latin1]{inputenc}
\usepackage{amssymb,amsmath}
\usepackage{verbatim}
\usepackage{color}
\usepackage{algorithm,algorithmic}
\usepackage{enumerate}
\usepackage{url}

\usepackage{geometry}

\usepackage{array}

\newtheorem{theorem}{Theorem}[section]
\newtheorem{lemma}[theorem]{Lemma}

\newtheorem{conjecture}[theorem]{Conjecture}
\newtheorem{corollary}[theorem]{Corollary}

\theoremstyle{definition}

\newtheorem{example}[theorem]{Example}
\newtheorem{assumption}[theorem]{Assumption}

\theoremstyle{remark}
\newtheorem{remark}[theorem]{Remark}

\numberwithin{equation}{section}

\newcommand\Z{\ensuremath{\mathbb Z}}
\newcommand\Q{\ensuremath{\mathbb Q}}\newcommand\R{\ensuremath{\mathbb R}}
\newcommand\C{\ensuremath{\mathbb C}}

\newcommand\GL{\operatorname{GL}}

\newcommand\M{\operatorname{M}}

\newcommand\PGL{\operatorname{PGL}}

\newcommand\SL{\operatorname{SL}}

\newcommand{\cH}{\mathcal{H}}
\newcommand{\fp}{{\mathfrak{p}}}

\newcommand{\fm}{{\mathfrak{m}}}
\newcommand{\fa}{{\mathfrak{a}}}
\newcommand{\fq}{{\mathfrak{q}}}

\def\cO{{\mathcal O}}

\newcommand{\mtx}[4]{\left(\begin{matrix}#1&#2\\#3&#4\end{matrix}\right)}

\newcommand{\smtx}[4]{\left(\begin{smallmatrix}#1&#2\\#3&#4\end{smallmatrix}\right)}

\def\M{\operatorname{M}}

\def\P{\mathbb P}

\def\fN{\mathfrak N}

\def\ol{\overline}

\newcommand{\ra}{\rightarrow}

\newcommand{\lra}{\longrightarrow}

\newcommand{\QQ}{\Q}
\newcommand{\PP}{\mathbb{P}}
\newcommand{\ZZ}{\Z}

\def\Xint#1{\mathchoice
{\XXint\displaystyle\textstyle{#1}}%
{\XXint\textstyle\scriptstyle{#1}}%
{\XXint\scriptstyle\scriptscriptstyle{#1}}%
{\XXint\scriptscriptstyle\scriptscriptstyle{#1}}%
\!\int}
\def\XXint#1#2#3{{\setbox0=\hbox{$#1{#2#3}{\int}$}
\vcenter{\hbox{$#2#3$}}\kern-.5\wd0}}

\begin{document}

\title[Darmon Points with Higher Conductor]{Elementary Matrix Decomposition and The Computation of Darmon Points with Higher Conductor}


\author{Xavier Guitart}
\address{ Universitat Polit\`ecnica de Catalunya, Barcelona \newline \indent Max Planck Institute for
Mathematics, Bonn}
\curraddr{}
\email{xevi.guitart@gmail.com}

\author{Marc Masdeu}
\address{Columbia University, New York}
\curraddr{}
\email{masdeu@math.columbia.edu}


\date{\today}

\dedicatory{}


\begin{abstract}
We extend the algorithm of~\cite{darmon-green} and~\cite{darmon-pollack} for computing $p$-adic Darmon points on elliptic curves to the case of composite conductor. We also extend the algorithm of~\cite{darmon-logan} for computing ATR Darmon points to treat curves of nontrivial conductor. Both cases involve an algorithmic decomposition into elementary matrices in congruence subgroups $\Gamma_1(\fN)$ for ideals $\fN$ in certain rings of $S$-integers. We use these extensions to provide additional evidence in support of the conjectures on the rationality of Darmon points.
\end{abstract}

\maketitle

\section{Introduction}
Let $E$ be an elliptic curve over $\Q$ of conductor $pM$, with $p$ a prime not dividing $M$. Let $K$ be a real
quadratic field in which $p$ is inert and all the primes dividing $M$ are split, and denote by $\cH_p=\P^1(K_p)\setminus
\P^1(\Q_p)$ the
$K_p$-points of the $p$-adic upper half plane. 

A construction of Darmon \cite{darmon-integration} associates to every
$\tau\in \cH_p$ a local point $P_\tau\in E(K_p)$, which is defined as a certain period of the modular form $f$
corresponding to $E$ under the Modularity Theorem. The points $P_\tau$ are conjectured to be rational over ring
class fields of $K$, and to behave in many aspects as Heegner points.

An algorithm for the effective calculation of  $p$-adic Darmon points was given in \cite{darmon-green} and was improved in~\cite{darmon-pollack}. Some $P_\tau$'s  were computed in concrete examples and checked to be
$p$-adically close to global points, providing extensive numerical evidence in support of the conjectures. However, due to the restrictions imposed by the algorithm, only  elliptic curves of prime conductor $p$ (that is,
with $M=1$) could be treated.

In the articles \cite{darmon-green} and \cite{darmon-pollack} it is crucial to assume that $M=1$ when applying the ``continued fraction trick'' (see \cite[p. 42]{darmon-green}) in order to transform certain semi-indefinite
integrals into double integrals. The present article provides a different procedure for performing this
step when $M > 1$ provided that $E$ satisfies a mild condition
(see Assumption \ref{assumption} for its precise statement).

In other words, we extend the algorithm of \cite{darmon-green} and \cite{darmon-pollack} to a much larger class of curves. As an application, we compute $p$-adic Darmon points on 
curves of composite
conductor and we check that they are close to global points, which provides new experimental evidence in support of the
validity of Darmon's construction in its stated generality.

The method for transforming semi-indefinite integrals into double integrals is based
on an algorithmic decomposition into elementary matrices in congruence subgroups $\Gamma_1(\fN)$ for ideals $\fN$ in certain rings of $S$-integers. This can be seen as an effective version of the congruence subgroup problem. In particular we improve on a theorem of Cooke-Weinberger (see Corollary~\ref{corollary:cw}).

In addition, essentially the same method of elementary matrix decompositions can also be applied to ATR
points, a different instance of Darmon points introduced in \cite[\S
8]{darmon-book}  and \cite[\S4]{darmon-logan} for elliptic curves over totally real fields. Although the construction of
ATR points is radically different (for instance, they are defined by means of complex periods), their explicit computation has some formal similarities with respect to the $p$-adic setting. In particular the methods used until the present have all used a ``continued fraction trick'' which in this case only applies to
curves of trivial conductor (cf. \cite[p. 108]{Ga-th} for a discussion of this issue). The present article
provides also a method for computing ATR Darmon points in curves of non-trivial conductor.

The rest of the article is organized as follows. In Section \ref{section: preliminaries} we introduce an algorithm for
computing elementary matrix decompositions in certain congruence subgroups. We state it in a level of
generality so that it can be applied both to the $p$-adic and the ATR setting. In Section \ref{section: p-adic} we
recall the definition of $p$-adic Darmon points, we discuss the algorithm for computing them in curves of composite
level and we include some tables of numerical computations performed using it. Finally, in Section
\ref{section: ATR} we briefly recall ATR points and we make explicit the method for computing them in curves of
non-trivial conductor, as well as a detailed example of a numerical verification of Darmon's conjecture in this case.

\subsection*{Acknowledgments} It is a pleasure to thank John Voight for helpful comments on elementary matrix
decompositions,  Robert Pollack for clarifying to us some details on his implementation of the overconvergent
modular symbols algorithm, and Henri Darmon for many valuable observations. Guitart wants to thank the Max Planck
Institute for Mathematics for their hospitality and financial support during his stay at the Institute, where part of
the present work has been carried out.

\section{Preliminaries: elementary matrix decomposition in $\Gamma_1$}\label{section: preliminaries}
Let $F$ be a number field with ring of integers $\cO$, and let $S$ be a finite set of places of $F$ containing the
archimedean ones. Let $\cO_S$ denote the subring of $F$ consisting of those elements whose valuation is non-negative at
all the places outside $S$.

For an ideal $\fN$ of $\cO_S$, we denote by $\Gamma_1(\fN)$  the subgroup of $\SL_2(\cO_S)$ defined as
\begin{equation}\label{eq: def of Gamma_1}
 \Gamma_1(\fN)=\left\{ \gamma \in \SL_2(\cO_S) \colon \gamma\equiv \mtx 1 * 0 1 \pmod \fN\right\},
\end{equation}
and by $E_{1,\fN}$  the subgroup of $\Gamma_1(\fN)$ generated by the \emph{elementary matrices} of the form
\begin{equation}\label{eq: elementary matrices}
 \mtx 1 0 y 1 \text{ with } y\in \fN \ \text{ and } \mtx 1 x 0 1 \ \text{ with } x \in \cO_S.
\end{equation}

If the group of units $\cO_S^\times$ is infinite, then   $\Gamma_1(\fN)=E_{1,\fN}$; see for instance 
\cite{Vas} for a proof of this result, and also for its relation with the Congruence Subgroup Problem. Therefore, when
$\cO_S^\times$ is infinite every matrix in $\Gamma_1(\fN)$ can
be factored into a product of elementary matrices of type \eqref{eq: elementary matrices}. However, the proof given in
\cite{Vas} is not explicit, so it does not give rise to an algorithm for
systematically performing such decomposition. In this section we see that the results and
techniques introduced in \cite{Cooke-Weinberger} can be adapted  to provide --assuming GRH-- such an algorithm, in the
particular case where  $F$ has at least a real archimedean place. Thus from now on  we will assume
that $\cO_S^\times$ is infinite and that $F$ has at least a real place. 

The following lemma is stated and proved in  \cite[Lemma 2.2 (b)]{Bass-Milnor-Serre} using the notation of Mennicke
symbols. For the
convenience of the reader we restate it in terms of explicit matrix formulas.

\begin{lemma}\label{lemma: case a equi 1 mod c}
 Let $\gamma=\smtx a b c d$ be an element in $\Gamma_1(\fN)$. Suppose that $c= u +ta$ for some unit $u\in
\cO_S^\times$ and some $t\in \cO_S$, and let $T$ be the matrix
\[
 T=\mtx{1}{0}{u(a-1)}{1}\mtx{1}{u^{-1}}{0}{1}\mtx{1}{0}{t(a-1)}{1}\mtx{1}{0}{-c}{1}.
\]
Then $T\gamma=\smtx 1 x 0 1$ for some $x\in \cO_S$. In particular,
\begin{equation}\label{eq: dec of gamma}
 \gamma=\mtx{1}{0}{c+t(1-a)}{1}\mtx{1}{-u^{-1}}{0}{1}\mtx{1}{0}{u(1-a)}{1}\mtx{1}{x}{0}{1
}.
\end{equation}

\end{lemma}
\begin{proof}
A direct computation shows that the first column of $T\gamma$ is 
$
 \left( \begin{smallmatrix}
         1\\ 0
        \end{smallmatrix}
\right).
$
Since $T\gamma$ belongs to $\Gamma_1(\fN)$, we see that $T\gamma=\smtx 1 x 0 1$ for some $x\in \cO_S$.
\end{proof}

Observe that, since $a-1$ and $c$ belong to $\fN$ , identity \eqref{eq: dec of gamma} already expresses $\gamma$ as a
product of elementary matrices of type \eqref{eq: elementary matrices}. The next step is to show that, assuming GRH, one
can reduce to the case where $a$ is congruent to a unit mod $c$ by multiplying $\gamma$ by an elementary matrix. Before
stating the result of \cite{Cooke-Weinberger} that grants this, we recall some terminology related to
ray class groups.

Let $\mathfrak m$ be
an ideal of $\cO$, with factorization into prime ideals of the form
\[
 \fm=\prod_{\fp\mid \fm}\fp^{m(\fp)}.
\]
Let $I^\fm$ be the multiplicative group of ideals of $\cO$ relatively prime to $\mathfrak m$, and let
\[
 F_{1,\fm}=\{x\in F\colon v_\fp(x-1)\geq m(\fp) \ \text{ for all } \fp\mid \fm\}.
\]
There is a natural map $i\colon F_{1,\fm}\ra I^\fm$ given by $x\mapsto (x)$. The quotient
\[
 C_\fm=I^\fm/i(F_{1,\fm})
\]
is the \emph{ray class group modulo $\fm$}. The following result is  \cite[Theorem 2.13]{Cooke-Weinberger}.
\begin{theorem}[Cooke--Weinberger]\label{th: Cooke-Weinberger}
 Assume GRH. Let $\fm$ be an ideal in $\cO$, and let $\alpha$ be an ideal class in $C_\fm$. Then the set of prime ideals
$\fq$ contained in $\alpha$ such that the reduction map
\[
 \cO_S^\times \lra (\cO_S/\fq\cO_S)^\times
\]
 is surjective has positive density.
\end{theorem}
From this we obtain the main result of this section.
\begin{theorem}\label{theorem: our main theorem}
 Let $\gamma=\smtx a b c d$ be an element of $\Gamma_1(\fN)$. Assuming GRH, the following algorithm terminates and 
computes an expression of $\gamma$ as a product of  elementary matrices.
\begin{enumerate}
\item \label{step 1} Iterate over the elements $\lambda\in \cO$ to find $\lambda$ such that $a'=a+\lambda c$
generates a prime ideal and
\begin{equation}\label{eq: units surjective}
 \cO_S^\times\lra \left(\cO_S/a'\cO_S\right)^\times \ \text{ is surjective.}
\end{equation}
\item Set $\gamma'=\smtx 1 \lambda 0 1 \gamma$, and let $\gamma'=\smtx{a'}{b'}{c'}{d'}$.

\item \label{step 3} Iterate over the elements $u\in \cO_S^\times$ until finding $u$ such
that 
\[
 c'\equiv u \pmod{a'}.
\]
\item Use Lemma \ref{lemma: case a equi 1 mod c} to find an expression of $\gamma'$ as a product of elementary
matrices
\end{enumerate}
\end{theorem}
\begin{proof} We can, and do, choose  $e\in \cO_S^\times$ such that   $ea$ and $ec$ belong to $\cO$. By applying Theorem
\ref{th: Cooke-Weinberger} with $\fm=(ec)$, we see that there exists a prime ideal $\fq$ in the
same class of $[(ea)]\in C_{\fm}$ such that $\cO_S^\times \ra \left(\cO/\fa'\cO_S \right)^\times$ is surjective. Since
$(ea)$ is integral and principal, we see that $\fq$ is also principal. Therefore, $\fq=(q)$ for some $q\equiv ea
\pmod{(ec)}$,
so that $q=ea+\lambda ec$ for some $\lambda\in \cO$. But $q$ and $a'=a+\lambda c$ generate the same ideal in $\cO_S$, so
that $\cO_S^\times\ra \left(\cO/a'\cO_S \right)^\times$ is surjective. This justifies that step \eqref{step
1} of the algorithm can be accomplished.

Now  the class of $c'$ in $\left(\cO_S/a'\right)^\times$ can be represented by some unit $u\in\cO_S^\times$, and this
justifies step $\eqref{step 3}$. 
\end{proof}

In the particular case $\fN=\cO_S$, the result in~\cite[Th. 2.14]{Cooke-Weinberger} asserts that any matrix $\gamma\in
\SL_2(\cO_S)$ can be expressed as a product of at most $7$
elementary matrices.   Theorem \ref{theorem: our main theorem} above and expression \eqref{eq: dec of gamma}
give the  following generalization to arbitrary ideals $\fN$, which also slightly improves the number of
elementary matrices needed to $5$.

\begin{corollary}
\label{corollary:cw}
Under the assumption of GRH, every matrix in $\Gamma_1(\fN)$ is a product of at most $5$ elementary
matrices of type \eqref{eq: elementary matrices}.
\end{corollary}

\section{Computation of $p$-adic Darmon points on curves with composite conductor}\label{section: p-adic}
In this section we explain the algorithm for computing $p$-adic Darmon points in curves of composite conductor and some
of the related computational issues. After briefly reviewing the definition of $p$-adic Darmon points in
\S\ref{subsection: p-adic darmon points},  in \S\ref{subsection: computation of semi-indefinite integrals} we describe
the algorithm to transform the semi-indefinite integrals appearing in the definition of these points into
double multiplicative
 integrals. In \S\ref{subsection: computation of definite double integrals} we explain an efficient way of
computing these double integrals, and finally in \S\ref{subsection: numerical computations} we comment on the
calculations that have been carried out in support of Darmon's conjecture.
\subsection{Review of $p$-adic Darmon points}\label{subsection: p-adic darmon points} Our presentation of the necessary
background and the definition of $p$-adic Stark--Heegner points follows closely  \cite[\S1]{darmon-pollack}, to
which we refer the
reader for more details and, in fact, for an excellent account of this material in the prime level case.

Let $E$ be an elliptic curve over $\Q$ of conductor $N=pM$, with $p$ a prime not dividing $M$. Let $f(z)=\sum_{n\geq 1}
a_ne^{2\pi i nz}$ be the weight two newform on $\Gamma_0(N)$ whose $L$-series coincides with that of $E$. The
coefficient $a_p$ is $1$ (resp. $-1$) if $E$ has split (resp. non-split) multiplicative reduction at
$p$.

Let $R$ be the order in $\M_2(\Z[1/p])$ consisting of matrices that are upper triangular modulo $M$, and let
$\Gamma=R^\times_1$ denote its group of units of determinant $1$. Let $K$ be a real quadratic field in which $p$ is
inert and all primes dividing $M$ are split. Set  $\cH_p=\P^1(K_p)\setminus \P^1(\Q_p)$, in which $\Gamma$ acts by
M\"obius transformations. The $p$-adic Darmon point construction yields a map
\[
 \begin{array}{ccc}
  \Gamma\setminus \cH_p&\lra& E(K_p)\\
\tau & \longmapsto & P_\tau,
 \end{array}
\]
given in terms of certain $p$-adic integrals, and whose definition ultimately relies on the  $\Z$-modular symbol
attached to $E$.

\subsubsection{Measures attached to modular symbols} If $V$ is a $\Z$-module, a \emph{$V$-valued modular symbol}
$\varphi$ is a map
\[\varphi\colon
\P^1(\Q)\times\P^1(\Q)\lra V,\quad (r,s)\mapsto \varphi\{r\ra s\}
\]
such that 
\[
 \varphi\{r\ra s\}+\varphi\{s\ra t\}=\varphi\{r\ra t\} \ \ \text{for all } r,s,t\in\P^1(\Q).
\]

For $w_\infty\in \{\pm 1\}$ we denote by $I_f\colon\P^1(\Q)\times \P^1(\Q)\ra \Z$
the $\Z$-valued modular symbol attached to $E$ and $w_\infty$. That is to say
\begin{equation*}
I_f\{r\ra s\}= 
\begin{cases}  \frac{1}{\Omega^+}\int_r^s \operatorname{Re} \omega_f & \text{if $w_\infty=+1$,}
\\
  \frac{1}{\Omega^-}\int_r^s \operatorname{Im} \omega_f & \text{if $w_\infty=-1$,}
\end{cases}
\end{equation*}
where $\omega_f=2\pi i f(z)dz$ and  $\Omega^+,\Omega^-\in \R^{>0}$ are the  unique periods with the property that the
map $I_f$ defined in this way takes values in $\Z$ and in no proper ideal of $\Z$. To simplify the exposition we assume
for the rest of the article that $\omega_\infty=1$, but the construction works very similarly for $\omega_\infty=-1$.

A \emph{$\Z$-valued measure on $\P^1(\Q_p)$} is a finitely additive function $\mu$ from the set of compact open subsets
of $\P^1(\Q_p)$ to $\Z$. In \cite{darmon-integration} Darmon attaches to each pair $r,s\in\P^1(\Q)$  a $\Z$-valued
measure $\mu_f\{r\ra s\}$ on $\P^1(\Q_p)$ with total measure $0$ by defining
\[
 \mu_f\{r\ra s\}(\gamma\Z_p):=I_f\{\gamma^{-1}r\ra \gamma^{-1}s\}.
\]
This is enough to define $\mu_f\{r\ra s\}$ for all compact open  $U\subset\P_1(\Q_p)$, because 
either $U$ or $\P^1(\Q_p)\setminus U$  are of the form $\gamma\Z_p$ for some $\gamma\in
\Gamma$.

\subsubsection{Double multiplicative integrals} If $h$ is a continuous function on $\P_1(\Q_p)$ and
$\mu$ is a measure on $\P^1(\Q_p)$ then the integral $\int_{\P^1(\Q_p)}h(x)d\mu(x)$ is defined by means of the Riemann
sum
\[
 \int_{\P^1(\Q_p)}h(x)d\mu(x)=\lim_{\mathcal U=\{U_\alpha\}}\sum_\alpha h(x_\alpha){\mu}(U_\alpha),
\]
where the limit is taken over increasingly finer coverings $\mathcal U$ of $\P^1(\Q_p)$ by compact open
subsets $U_\alpha$, and $x_\alpha$ is any point in $U_\alpha$. If $\mu$ is $\Z$-valued the \emph{multiplicative 
 integral} can be defined by replacing the Riemann sum by a Riemann product:
\[
 \Xint\times_{\P^1(\Q_p)}h(x)d\mu(x)=\lim_{\mathcal U=\{U_\alpha\}}\prod_\alpha h(x_\alpha)^{\mu(U_\alpha)}.
\]

For $\tau_1,\tau_2\in \cH_p$ and
$r,s\in\P^1(\Q_p)$ Darmon defines a $K_p$-valued \emph{double integral} as
\[
 \int_{\tau_1}^{\tau_2}\int_r^s\omega_f=\int_{\P^1(\Q_p)}\log \left( \frac{x-\tau_2}{x-\tau_1}\right)d\mu_f\{r\ra
s\}(x),
\]
where $\log$ denotes a fixed branch of the $p$-adic logarithm. Since $\mu_f\{r\ra s\}$ takes values in $\Z$ a
$K_p^\times$-valued \emph{double multiplicative integral} can be defined as
 \[
 \Xint\times_{\tau_1}^{\tau_2}\int_r^s\omega_f=\Xint\times_{\P^1(\Q_p)}\left( 
\frac{x-\tau_2}{x-\tau_1}\right)d\mu_f\{r\ra s\}(x).
\]
 These two integrals are related by the formula
\[
 \int_{\tau_1}^{\tau_2}\int_r^s\omega_f=\log\left( \Xint\times_{\tau_1}^{\tau_2}\int_r^s\omega_f  \right).
\]

\subsubsection{Semi-indefinite integrals}\label{subsubsec: semi-definite integrals} The  double multiplicative integral
satisfies the usual additivity properties
with respect to the limits, as well as the  $\Gamma$-invariance property
\[
 \Xint\times_{\gamma\tau_1}^{\gamma\tau_2}\int_{\gamma r}^{\gamma s}\omega_f=\Xint\times_{\tau_1}^{\tau_2}\int_{
r}^{ s}\omega_f\ \  \ \text{for all } \gamma\in \Gamma.
\]
Therefore, it gives rise to a group homomorphism
\[
 \operatorname{Int^\times}\colon
\left(\operatorname{Div}^0(\cH_p)\otimes\operatorname{Div}^0(\P^1(\Q))\right)_\Gamma\lra K_p^\times,
\]
where the subscript $\Gamma$ denotes the subgroup of $\Gamma$-coinvariants. Let $\Z[\Gamma]$ denote the group ring of
$\Gamma$ and let $I_\Gamma$ denote the augmentation ideal, defined by the
exact sequence
\[
 0\lra I_\Gamma\lra \Z[\Gamma]\lra \Z\lra 0.
\]
Tensoring with $I_\Gamma$ over $\Z$ and taking $\Gamma$-coinvariants gives
\[
 0\lra K_\Gamma\lra (I_\Gamma\otimes I_\Gamma)_\Gamma\stackrel{r}{ \lra} (\Z[\Gamma]\otimes I_\Gamma)_\Gamma\lra
(I_\Gamma)_\Gamma\lra 0.
\]
Since $I_\Gamma$ is generated over $\Z$ by elements of the form $\gamma-1$,  choosing base points $\tau\in\cH_p$ and
$x\in\P^1(\Q)$ one can define integration maps
\[
 \operatorname{Int}_{\tau,x}^\times\colon (I_\Gamma\otimes I_\Gamma)_\Gamma\lra K_p^\times
\]
determined by
\[
 \operatorname{Int_{\tau,x}^\times}((\gamma_0-1)\otimes(\gamma_1-1))=\int_\tau^{\gamma_0\tau}\int_x^{\gamma_1
x}\omega_f, \ \ \ \text{ for } \gamma_0,\gamma_1\in\Gamma.
\]
Letting $\Lambda=I_{\tau,x}(K_\Gamma)\subset K_p^\times$ yields a well defined map
\begin{equation}\label{eq: integration on im r}
 \operatorname{Int}_{\tau,x}^\times\colon \operatorname{Im}(r)\lra K_p^\times/\Lambda.
\end{equation}
The group $(I_\Gamma)_\Gamma\simeq \Gamma_{ab}$ is finite, say of exponent $e_\Gamma$. If  $y=\gamma x\in\P^1(\Q)$ is in
the same
$\Gamma$-orbit as $x$, define 
\[
 \Xint\times^{\tau}\int_x^y e_\Gamma\omega_f:=\operatorname{Int}_{\tau,x}^\times(e_\Gamma\cdot 1\otimes
(\gamma-1))\in K_p^\times/\Lambda.
\]
This \emph{semi-indefinite integral} satisfies the following properties:
\begin{enumerate}
 \item\label{property 1} $\displaystyle\Xint\times^\tau\int_r^s e_\Gamma\omega_f \times \Xint\times^\tau\int_s^t 
e_\Gamma \omega_f=
\Xint\times^\tau\int_r^t  e_\Gamma\omega_f,\ \ \text{for all } \tau\in \cH_p,\ r,s,t\in \Gamma x$; 
 \item \label{property 2} $\displaystyle\Xint\times^{\tau_2}\int_r^s  e_\Gamma\omega_f \div
\Xint\times^{\tau_1}\int_r^s e_\Gamma\omega_f=
\Xint\times_{\tau_1}^{\tau_2}\int_r^s
 e_\Gamma\omega_f,  \text{ for all } \tau_1,\tau_2\in \cH_p,\ r,s\in \Gamma x$;
\item \label{property 3}$\displaystyle\Xint\times^{\gamma\tau}\int_{\gamma r}^{\gamma s}  e_\Gamma\omega_f
=\displaystyle\Xint\times^{\tau}\int_{r}^{ s} e_\Gamma \omega_f$, for all $\gamma\in\Gamma$.
\end{enumerate}

\subsubsection{Darmon points} Let $q$ denote the $p$-adic period of $E$ and let $\Phi_{\rm{Tate}}\colon
K_p^\times/q^\Z\ra E(K_p)$ be Tate's uniformization map.

Given $\tau\in \cH_p$, let $\cO_\tau$ be the ring of matrices in $R$ that have the vector $(\tau,1)$ as
eigenvector, which is isomorphic to a $\Z[1/p]$ order of $K$. Let $H_\tau$ denote its
ring class
field which we can, and do, view as a subfield of $K_p$ by choosing a prime of $H_\tau$ above $p$. The stabilizer
$\Gamma_\tau$ of $\tau$ in
$\Gamma$ is a cyclic group of infinite order isomorphic to $\cO_{\tau,1}^\times/\langle \pm 1\rangle $. Let
$\gamma_\tau$ be a generator of $\Gamma_\tau$, and define
\[
 J_\tau=\Xint\times^\tau\int_\infty^{\gamma_\tau\infty}e_\Gamma\omega_f.
\]
\begin{conjecture}[Darmon]
\label{conjecture:darmon}
The local point $P_\tau=\Phi_{\rm{Tate}}(J_\tau)$ belongs to $E(H_\tau)$. 
\end{conjecture}

\subsection{Computation of semi-indefinite integrals}\label{subsection: computation of semi-indefinite integrals}
In order to effectively compute the points $J_\tau$ one needs to compute the semi-indefinite integrals
\begin{equation}\label{eq: semi-indefinite integral}
 \Xint\times^\tau\int_\infty^{\gamma_\tau\infty}e_\Gamma\omega_f.
\end{equation}
The method used in \cite{darmon-green} and \cite{darmon-pollack} boils down to using properties \eqref{property 1},
\eqref{property 2} and \eqref{property 3} of semi-indefinite integrals to express them in terms of double integrals,
which can be effectively computed either  via Riemann products as in \cite{darmon-green} or, more efficiently, via
overconvergent modular symbols as in \cite{darmon-pollack}. 

The algorithm for expressing semi-indefinite integrals in terms of double integrals of \cite{darmon-green} and
\cite{darmon-pollack} is based on the continued fraction algorithm, and it only works under the assumption that $M=1$
(i.e., that $E$ has  conductor equal to $p$). In this section we introduce an algorithm that, assuming GRH and under
some mild assumptions on $E$, works for all levels $M$. In Section \ref{subsection: computation of definite double
integrals} we will see that the resulting definite double integrals obtained by this method can also be computed using
overconvergent modular symbols, by suitably adapting the techniques of \cite{darmon-pollack}.

From now on we assume that the modular form $f$ satisfies the following condition, which is non-vacuous only when $M$ is prime.
\begin{assumption}\label{assumption}
 There is a $d>1$, $d\mid M$ such that  $f$ has eigenvalue $1$ with respect to the Atkin-Lehner
operator $W_d$. 
\end{assumption}

\begin{remark}\label{rmk: ass} Under Assumption \ref{assumption}, the double multiplicative integral is also invariant
under the matrix
$w_d=\smtx{0}{1}{-d}{0}$:
\begin{equation*}
 \Xint\times_{w_d\tau_1}^{w_d\tau_2}\int_{w_dr}^{w_ds}e_\Gamma\omega_f=\Xint\times_{\tau_1}^{\tau_2}\int_{r}^{s}
e_\Gamma\omega_f.
\end{equation*}
Let $\tilde \Gamma$ be the subgroup of $\PGL_2(\Q)$ generated by  $\Gamma$ and $w_d$. Then, by replacing $\Gamma$ by
$\tilde\Gamma$ in the argument of Section \ref{subsection: computation of definite double integrals} one can extend
the definition of the semi-indefinite integrals $\Xint\times^\tau\int_r^se_\Gamma\omega_f$  to all pairs
$r,s\in\P^1(\Q)$ lying in the
same $\tilde \Gamma$-orbit.
\end{remark}

Let $\Gamma_1$ be the congruence subgroup of $\SL_2(\Z[1/p])$ defined as
\[
 \Gamma_1=\{\gamma\in\Gamma\colon \gamma \equiv \smtx 1 * 0 1 \pmod M\}\subset \Gamma.
\]
Using the properties of the multiplicative integral it is easy to see that
\[
 (J_\tau)^m=\Xint\times^\tau\int_{\infty}^{\gamma_\tau^m\infty}\omega_f.
\]
Therefore, replacing  $P_\tau$ by a multiple of it if necessary we can always assume that  $\gamma_\tau$ belongs to
$\Gamma_1$ (but see also Remark \ref{rmk}).

Observe that $\Gamma_1$  is one of the groups treated in Section \ref{section: preliminaries}.
Indeed, with the
notation as in that section, if we let $F=\Q$, $S=\{\infty,p\}$ and $\fN=M\cdot \Z[1/p]$ we
have that $\Gamma_1=\Gamma_1(\fN)$. In particular, the algorithm described in Theorem \ref{theorem: our main theorem}
gives an effective method (under our running assumption of GRH) for computing a decomposition of $\gamma_\tau$ of the
form
\begin{equation}\label{eq: decomposition of gamma}
 \gamma_\tau=U_1 L_1 U_2L_2 U_3,
\end{equation}
where the  matrices $U_i$ and $L_i$ are of the form
\[
 U_i=\mtx{1}{x_i}{0}{1}\  \text{for some } x_i\in\Z[1/p],\ \
  L_i=\mtx{1}{0}{y_i}{1} \text{ for some }
y_i\in M\cdot \Z[1/p].
\]
In particular, $L_i$ and $ U_i$ belong to $\Gamma$. Then, for $G\in \Gamma$ we have that 
\begin{equation}\label{eq: upper triangular}
 \Xint\times^\tau\int_\infty^{U_i
G\infty}e_\Gamma\omega_f=\Xint\times^{U_i^{-1}\tau}\int_\infty^{G\infty}e_\Gamma\omega_f
\end{equation}
and
\begin{equation}\label{eq: lower triangular}
\begin{split}
  \Xint\times^\tau\int_\infty^{L_i
G\infty}e_\Gamma\omega_f
&=\Xint\times^\tau\int_\infty^0e_\Gamma\omega_f\times\Xint\times^\tau\int_0^{L_iG\cdot\infty}
e_\Gamma\omega_f\\&=\Xint\times^\tau\int_\infty^0e_\Gamma\omega_f \times
\Xint\times^{L_i^{-1}\cdot \tau}\int_0^{G\cdot \infty}\omega_f^+\\&=\Xint\times^{
\tau}\int_\infty^{0}e_\Gamma\omega_f\times\Xint\times^{L_i^{-1}\cdot
\tau}\int_{0}^{\infty}e_\Gamma\omega_f\times\Xint\times^{L_i^{-1}\cdot
\tau}\int_{\infty}^{G\cdot \infty}e_\Gamma\omega_f\\&=\Xint\times_\tau^{L_i^{-1}\cdot
\tau}\int_{0}^{\infty}e_\Gamma\omega_f\times\Xint\times^{L_i^{-1}\cdot
\tau}\int_{\infty}^{G\cdot \infty}e_\Gamma\omega_f.
\end{split}
\end{equation}
In view of decomposition \eqref{eq: decomposition of gamma},  repeated application of  \eqref{eq: upper triangular} and
\eqref{eq: lower triangular} transforms the semi-indefinite integral \eqref{eq: semi-indefinite integral} into a
product of double multiplicative integrals. Observe that Assumption \ref{assumption} is crucial in \eqref{eq: lower
triangular} because of the integral $\Xint\times^\tau\int_\infty^0 e_\Gamma\omega_f$. Indeed, the cusps $0$ and
$\infty$ are never in the same $\Gamma$-orbit if $M>1$, but $w_d\infty=0$ so they are in the same
$\tilde\Gamma$-orbit and thanks to Remark \ref{rmk: ass} the integral $\Xint\times^\tau\int_\infty^0e_\Gamma\omega_f$
is well defined.

\begin{remark}\label{rmk}
 As we already mentioned, one can overcome  the fact that   $\gamma_\tau=\smtx a b c d$  does not generally
belong to $\Gamma_1$ by computing an appropriate power $(J_\tau)^m$. In some cases one can apply an
alternative procedure instead, which turns out to be more convenient in the actual computations and it allows for the
computation of $J_\tau$ itself. Namely, if $a\equiv p^n \pmod M$ for some integer $n$, then the
matrix $g=\mtx{p^{-n}}{0}{0}{p^n}$ belongs to $\Gamma$ and
\[
 J_\tau=\Xint\times^\tau\int_\infty^{\gamma_\tau\infty}e_\Gamma\omega_f=\Xint\times^{g\tau}\int_\infty^{
g\gamma_\tau\infty } e_\Gamma\omega_f
\]
with $g\gamma_\tau$ belonging to $\Gamma_1$.

\end{remark}

\subsection{Computation of the definite double integrals}\label{subsection: computation of definite double integrals}

As we have seen in Subsection~\ref{subsection: computation of semi-indefinite integrals} the  computation of
the period $J_\tau$ is reduced to products of integrals of the form
\begin{equation}
\label{eq:multintegral-compute}
\Xint\times_{\PP^1(\QQ_p)}\left(\frac{x-\tau_2}{x-\tau_1}\right)d\mu_f\{r\to s\}(x).
\end{equation}
Write $\mu$ for the measure $\mu_f\{r\to s\}$, and consider a decomposition of $\PP^1(\QQ_p)$ into a disjoint union of $L$
open balls of the form
\begin{equation}
\label{eq:decomp-P1}
\PP^1(\QQ_p)=\bigcup_{i=1}^L g_i\cdot\ZZ_p,\quad g_i=\mtx{a_i}{b_i}{c_i}{d_i}\in\GL_2(\Q),
\end{equation}
which will be fixed later. This yields in turn a decomposition
\[
\Xint\times_{\PP^1(\QQ_p)}\left(\frac{x-\tau_2}{x-\tau_1}\right)d\mu(x)=\prod_{i=1}^L \Xint\times_{g_i\ZZ_p}
\left(\frac{x-\tau_2}{x-\tau_1}\right)d\mu(x).
\]
Fix such an $i$ and let $g = g_i$ be the corresponding matrix, written $g=\smtx abcd$. We are thus reduced to calculating
\[
\Xint\times_{g\ZZ_p} \left(\frac{x-\tau_2}{x-\tau_1}\right)d\mu(x).
\]
Apply the change of variables $x=g\cdot t$ to get
\[
\Xint\times_{g\ZZ_p} \left(\frac{x-\tau_2}{x-\tau_1}\right)d\mu(x) = \Xint\times_{\ZZ_p} \left(\frac{g\cdot t-\tau_2}{g\cdot t-\tau_1}\right)d\mu(g\cdot t).
\]
Let $\log_p$ be the unique homomorphism $K_p^\times\to K_p$ such that $\log_p(1-t)=-\sum_{n=1}^\infty t^n/n$ and 
$\log_p(p)=0$. It is surjective, with kernel
\[
\ker \left(\log_p\colon K_p^\times\to K_p \right)= p^\ZZ\times \mathbf{U},
\]
where $\mathbf{U}$ is the group of roots of unity in $K_p^\times$. Suppose that we can express the integrand as a power
series in $t$ of the form
\begin{equation}
\label{eq:pseriesexpansion}
\left(\frac{g\cdot t-\tau_2}{g\cdot t-\tau_1}\right) = \alpha_0\left(1+\sum_{n=1}^\infty \alpha_np^n t^n\right),
\end{equation}
with $\alpha_n$ belonging to $\ZZ_p$ for all $n\geq 1$. Then the expression in~\eqref{eq:pseriesexpansion} converges
for $t\in\ZZ_p$ and is
constant modulo $p^{v_p(\alpha_0)+1}$. Therefore the expression of~\eqref{eq:multintegral-compute} can be determined
modulo this power of $p$ by performing $L$ evaluations. The
logarithm $\log_p(J_\tau)$ of the  period is evaluated by noting that
\[
\log_p\left(\alpha_0(1+\sum \alpha_np^nt^n)\right) = \log_p\alpha_0 - \sum_{n=1}^\infty \frac{\alpha_n(-p)^nt^n}{n}.
\]
Interchanging the infinite sum with the integral, finding $\log_p(J_\tau)$ boils down to calculating
\[
\int_{\ZZ_p} t^n d\mu(g\cdot t) = \int_{g\ZZ_p} (g^{-1}\cdot t)^nd\mu(t),
\]
which is the $n$th moment of $\mu$ at $g\ZZ_p$. This data can be efficiently computed in time polynomial in the number of $p$-adic digits of precision, thanks to the methods of Darmon and Pollack (see for instance~\cite[display 23]{darmon-pollack}). Finally, one recovers the original period via the formula
\[
J_\tau = p^{v_p(J_\tau)}\cdot \zeta\cdot \exp_p(\log_p J_\tau),
\]
where $\zeta$ is the Teichmuller lift of the unit part modulo $p$ of the multiplicative integral. Note also that
$v_p(J_\tau)$ is the sum of the valuations of the $\alpha_0$ appearing in the decomposition~\eqref{eq:decomp-P1}. In 
order to find the power series in~\eqref{eq:pseriesexpansion}, we calculate:
\begin{align*}
\frac{gt-\tau_2}{gt-\tau_1} &= \frac{\frac{at+b}{ct+d} -\tau_2}{\frac{at+b}{ct+d}-\tau_1}\\
&= \frac{(a-c\tau_2)t+(b-d\tau_2)}{(a-c\tau_1)t+(b-d\tau_1)}\\
&= \frac{b-d\tau_2}{b-d\tau_1}\frac{1+t\frac{c\tau_2-a}{d\tau_2-b}}{1+t\frac{c\tau_1-a}{d\tau_1-b}}\\
&= \frac{b-d\tau_2}{b-d\tau_1}\left( 1 + \frac{\ol g\tau_1-\ol g\tau_2}{\ol g \tau_1}\sum_{i=1}^\infty (-1)^n(\ol g\tau_1)^nt^n\right)
\end{align*}
where $\ol g$ is the matrix
\[
\ol g=\mtx{0}{-1}{1}{0} g^{-1}.
\]
Note that for any two matrices $g,h$ we have $\ol{hg}=\ol{g} h^{-1}$. Also, note that for any choice of $h$, a decomposition
\[
\PP^1(\Q_p)=\bigcup_i g_i\cdot\Z_p
\]
gives rise to another decomposition
\[
\PP^1(\Q_p)=\bigcup_i (hg_i)\cdot\Z_p.
\]
Therefore by choosing an appropriate $h\in\GL_2(\Q_p)\cap M_2(\Z)$ we can assume that $v_p(\tau_1-a) = 0$ for all
$a=0,1,\ldots, p-1$ and that $v_p(\tau_2)\geq 0$.

Note that in order to obtain a power series as in~\eqref{eq:pseriesexpansion}, the matrices $g=\smtx abcd$ that we consider should satisfy
\begin{align}
\label{eq:valuation}
v_p\left(\frac{c\tau_2-a}{d\tau_2-b}\right)&\geq 1.
\end{align}
The conditions on $\tau_1$ and $\tau_2$ imply that the matrix corresponding to the contribution of
$\PP^1(\QQ_p)\setminus\ZZ_p$ satisfies \eqref{eq:valuation}, and we  concentrate on the integral on $\ZZ_p$. Let $r\geq
1$ be the largest integer such that $\tau_2$ is congruent to some integer modulo $p^r$. Let $t_2\in\ZZ$ be a
representative for the class of $\tau_2 \pmod{p^r}$. Write also $t^{(i)}$ for
the representative of $t_2\pmod{p^i}$ in the range $0,\ldots,p^i-1$. We can then use the decomposition given by the
matrices $g$ in the set $G = \cup_{i=1}^{r+1} G_i$, where
\[
G_i=\left\{ \mtx{p^i}{t^{(i)}+bp^{i-1}}{0}{1}\mid b=1,\ldots,p-1\right\},\quad G_{r+1} = \left\{ \mtx{p^{r+1}}{p^rb}{0}{1}\mid b=0,\ldots,p-1\right\}.
\]
This decomposition consists of exactly $p + 1 + r(p-1)$ opens.
\subsection{Numerical computations}\label{subsection: numerical computations}
To test our methods we have written a Sage implementation of the above algorithms, modifying an existing implementation written by Robert Pollack which in turn adapted part of the code originally written in Magma by Darmon and Pollack (\cite{darmon-pollack}). The code can be found in the second author's webpage.

Given an elliptic curve  of conductor $N = pM$ and a quadratic field $K$, the code finds all the optimal embeddings of
level $N$ of $K$ into $M_2(\ZZ[\frac{1}{p}])$,  and computes the Stark-Heegner period corresponding to the fixed point
of $K$ acting on $\cH_p$ via each embedding to a prescribed precision. The Tate parametrization yields the
coordinates of the Stark-Heegner point on $E(K_p)$ which are then recognized as algebraic coordinates using standard
routines.

Apart from gathering numerical evidence in support of Darmon's conjecture, it is also worth remarking that the relative
large height of the points thus found would make it impossible to find them using naive point search methods. This is,
therefore, the only known method to finding points of infinite order on such curves.

The rest of this subsection contains the evidence that we have
collected in support of Conjecture~\ref{conjecture:darmon}. Although we do not intend to be exhaustive, we provide
examples of curves of small composite conductor which satisfy Assumption~\ref{assumption} and for
which we have been able to positively test the conjecture. For each of these curves we consider all the real
quadratic fields $K$ of discriminant $D<200$ allowed by the splitting conditions on $p$ and $M$; for each
such field, we consider $\tau\in \cH_p$ such that $H_\tau$ equals the Hilbert class field of $K$, and we are
able to recognize in all the cases $P_\tau$ as an algebraic point defined over the Hilbert class
field of $K$. For those fields with nontrivial class group, we give the relative minimal polynomial $h_D$ of the
$X$-coordinate of the point.
\begin{table}[H]
\setlength{\extrarowheight}{5pt}
\label{table:curve15a1}
\input{curve15.tex}
\caption{Points on elliptic curve 15A1, with $p=5$}
\end{table}
\begin{table}[H]
\setlength{\extrarowheight}{5pt}
\label{table:curve21a1}
\input{curve21.tex}
\caption{Points on elliptic curve 21A1, with $p=3$}
\end{table}
\begin{table}[H]
\label{table:curve33a1}
\setlength{\extrarowheight}{5pt}
\input{curve33}
\caption{Points on elliptic curve 33A1, with $p=11$}
\end{table}
\begin{table}[H]
\label{table:curve35a1}
\setlength{\extrarowheight}{5pt}
\input{curve35}
\caption{Points on elliptic curve 35A1, with $p=7$}
\end{table}
\begin{table}[H]
\setlength{\extrarowheight}{5pt}
\label{table:curve51a1}
\input{curve51}
\caption{Points on elliptic curve 51A1, with $p=3$}
\end{table}
\begin{table}[H]
\setlength{\extrarowheight}{5pt}
\label{table:curve105a1}
\input{curve105}
\caption{Points on elliptic curve 105A1, with $p=3$}
\end{table}

\section{Computation of ATR Darmon points on curves of non-trivial conductor}\label{section: ATR}
In Section \ref{section: p-adic} we have seen that the algorithm of Theorem \ref{theorem: our main theorem} can be used
 in the computation of the semi-indefinite integrals entering the definition of $p$-adic Darmon points. It is a
substitute for the continued fractions trick of 
\cite{darmon-green} and
\cite{darmon-pollack}.

There is another type of Darmon points, called ATR, whose definition also relies in certain
semi-indefinite integrals. Although the framework is different (e.g., they are points on elliptic curves over number
fields, and the integrals are complex instead of $p$-adic), the formal properties satisfied by the semi-indefinite
integrals are the same in both settings. In the ATR  case, the continued fraction algorithm over number
fields had been used for computing ATR points on curves with trivial conductor (cf. \cite{darmon-logan},
\cite{guitart-masdeu}). Using a method analogous
to that of Section \ref{subsection:
computation of semi-indefinite integrals},  Theorem \ref{theorem: our main theorem} also allows for the computation of
ATR points on curves with non-trivial conductor.

To be more precise, let $F$ be a real quadratic number field of narrow class number $1$ and let $\cO$ denote its ring
of integers. Let
$E$ be
an elliptic curve over  $F$ of conductor $\fN$, and let  $\Gamma$ be the congruence subgroup consisting
of matrices in $\SL_2(\cO)$ that are upper triangular modulo $\fN$. Assuming that $E$ is modular, there is a Hilbert
modular form $f$ of parallel weight two and level $\fN$ whose $L$-series coincides with that of $E$. Let $\omega_f$
denote the corresponding $\Gamma$-invariant differential $2$-form on $\cH\times \cH$ (with $\Gamma$ acting on it via the
product of the two embeddings of $F$ into $\R$).

The following is analogous to Assumption \ref{assumption}.
\begin{assumption}
 There exists an ideal $\mathfrak D\mid \fN$ such that $f$ has eigenvalue $1$ for the Atkin--Lehner operator
$W_{\mathfrak D}$. 
\end{assumption}
Let $\tilde \Gamma$  be the subgroup of $\PGL_2(F)$ generated by
$\Gamma$ and the Atkin--Lehner matrix corresponding
to $W_{\mathfrak D}$. The previous assumption guarantees that
$0,\infty\in\P^1(F)$ are $\tilde \Gamma$-related.

Let $K$ be a quadratic ATR extension of $F$; i.e., a quadratic extension of $F$
that has exactly one non-real
archimedean place. Suppose that all primes dividing $\fN$ are split in $K$. We refer the reader to \cite{darmon-logan}
and \cite[\S8]{darmon-book} for the definition of the semi-indefinite integrals in this setting. We will just mention
that they are expressions of the form
\[
 \int^\tau\int_x^y\omega_f^+\in\C/\Lambda_f,
\]
where  $\tau\in\cH$, $x,y\in \P^1(F)$ are in the same $\tilde\Gamma$-orbit, $\Lambda_f$ is a certain period
lattice that depends on $f$, and
$\omega_f^+$ is a non-holomorphic differential easily related to $\omega_f$. They satisfy analogous properties
to those of $p$-adic semi-indefinite integrals; namely
\begin{enumerate}[(i)]
 \item\label{enum:p1} $\int^{\gamma\tau}\int_{\gamma r}^{\gamma s} \omega_f^+=\int^{\tau}\int_{r}^{s}
\omega_f^+$ for all $\gamma\in \tilde\Gamma,\ \ r,s\in \tilde\Gamma x$,
\item\label{enum:p2} $\int^{\tau}\int_{r}^{s} \omega_f^+ + \int^{\tau}\int_{r}^{s}
\omega_f^+=\int^{\tau}\int_{r}^{s} \omega_f^+$,  for all $\tau\in\cH,\ \ r,s\in \tilde\Gamma x$,
\item\label{enum:p3} $\int^{\tau_2}\int_{r}^{s} \omega_f^+-\int^{\tau_1}\int_{r}^{s}
\omega_f^+=\int_{\tau_1}^{\tau_2}\int_{r}^{s} \omega_f^+$,   for all $\tau_1,\tau_2\in\cH,\ \ r,s\in \tilde\Gamma x$.
\end{enumerate}

ATR Darmon points are given by expressions of the form
$ P_\tau=\Phi(\lambda\int^\tau\int_\infty^{\gamma_\tau\infty}\omega_f^+),$ where $\gamma_\tau$ belongs to $\Gamma$, 
$\Phi$ is the complex
uniformization map $\Phi\colon \C/\Lambda_E\ra E(\C)$, and $\lambda$ is a period conjecturally relating $\Lambda_E$ and
$\Lambda_f$.
At the cost of replacing $P_\tau$ by a multiple of it, we can
assume that $\gamma_\tau$ actually belongs to $\Gamma_1(\fN)$ (the notation is as in \eqref{eq: dec of gamma}, with $S$
equal to the archimedean places of $F$). Then the algorithm of Theorem \ref{theorem: our main theorem} computes a
decomposition of $\gamma_\tau$ of the form
\begin{equation}
 \gamma_\tau=U_1 L_1 U_2L_2 U_3,
\end{equation}
where the matrices $U_i$ and $L_i$ are of the form
\[
 U_i=\mtx{1}{x_i}{0}{1}\  \text{for some } x_i\in\cO, \ \
  L_i=\mtx{1}{0}{y_i}{1} \text{ for some }
y_i\in \fN.
\]
In particular, they belong to $\Gamma$ and the same expressions of \eqref{eq: upper triangular} and \eqref{eq: lower
triangular} (changing the multiplicative by the additive notation) express
$\int^\tau\int_\infty^{\gamma_\tau\infty}\omega_f^+$ in terms of usual double integrals of the form
$\int_{\tau_1}^{\tau_2}\int_0^\infty\omega_f^+$, which in principle can be evaluated by integrating the Fourier
expansion of $\omega_f^+$.

It is worth remarking that, although the method described above certainly expresses the semi-indefinite integrals in
terms of definite ones,  the running time to directly compute the resulting double integrals to a useful accuracy often
turns out to be too high.  The problem is that if the limits of integration are too close
to the real axis, then the number of Fourier coefficients needed to sum the series to an accurate precision is too
high. 

These kind of computational difficulties seem to be inherent to the ATR setting, as they were also to some extent
present in the initial work of Darmon and Logan \cite{darmon-logan}.  In \cite{guitart-masdeu} some methods for
accelerating the
computation of double integrals in the trivial level case were introduced. The authors believe that similar
techniques
should be applied to the non-trivial level setting in order to  perform a systematical calculation similar to the one of
Section \ref{subsection: numerical computations}. 

In spite of  this, in some simple examples it is possible to directly compute the integrals provided by Theorem
\ref{theorem: our main theorem},
and hence to compute approximations to the ATR points. The following is an example of this, which we detail because it
provides numerical evidence of the validity of Darmon's conjecture in elliptic curves of non-prime conductor.

\begin{example}
Let $F=\Q(\sqrt{5})$ and let $E$ be the curve
\[
 y^2+xy+\omega y = x^3-(\omega+1)x^2-(30\omega+45)x-(111\omega+117), \ \ \omega=\frac{1+\sqrt 5}{2}.
\]
The conductor of $E$ is  $\fN=(\sqrt{5}+6)$, which has norm $31$. This curve was previously
considered in \cite{Gr-th} and \cite{Ga-th} (but note the typo in the displayed equation in both references). Let
$\alpha=1-\sqrt{5}$ and let $K=F(\sqrt{\alpha})$. The embedding $\varphi\colon K\hookrightarrow \M_2(F)$ sending
$\omega$ to the matrix 
\[
W=\mtx{3-\omega}{-1}{8-3\omega}{-3+\omega}
\]
is an optimal embedding of level $\fN$. Under the embedding $F\hookrightarrow \R$ sending $\sqrt{5}$ to the positive
square root of $5$, the fixed point of $W$ acting on $\C^\times$ is 
\[
 \tau=0.439291418991 + i\cdot 0.353408129753,
\]
and the image of the unit $(-3+2\omega)\omega +4-3\omega\in\cO_K^\times$ under $\varphi$ is
\[
 \gamma_\tau=\mtx{-4+3\omega}{2-2\omega}{-22+16\omega}{12-9\omega}.
\]
The ATR point attached to (the maximal order of) $K$ is 
\[
 J_\tau=\int^\tau\int_\infty^{\gamma_\tau\infty}\omega_f^+.
\]
The determinant of $\gamma_\tau$ is $\omega+1$, which is a unit. However its upper left entry is not congruent to $1$
modulo $\fN$. If we let $u=-\sqrt{5}-2$, which is a  fundamental unit of $F$, then $-4+3\omega\equiv u
\pmod \fN$. This implies that the matrix $\gamma_\tau'=\smtx{u^{-1}}{0}{0}{u}\gamma_\tau$ has its upper left entry
congruent to $1$ modulo $\fN$. We can work with $\gamma_\tau'$ because 
\begin{equation}\label{eq: ATR point example}
 J_\tau=\int^{u^{-2}\tau}\int_\infty^{\gamma_\tau'\infty}\omega_f^+.
\end{equation}
Observe that $\det(\gamma_\tau')=w+1$, which is a unit. The algorithm of Theorem \ref{theorem: our main theorem} works
for invertible matrices,
not just determinant $1$ matrices, and it gives the following decomposition of $\gamma_\tau'$:
\[
 \smtx{1}{1-w}{0}{1}\smtx{1}{0}{118739 -73384\omega}{1}\smtx{1}{46368+
75025\omega}{0}{1}\smtx{1}{0}{-5431444+\omega 3356817}{1}\smtx{1}{-37268-60300\omega}{0}{1+w}.
\]
We use this decomposition to transform \eqref{eq: ATR point example} into a sum of usual double integrals. The
resulting integrals have limits not too close to the real axis (the smallest imaginary part is $\simeq 0.011$). 
Integrating the Fourier series with coefficients $a_{\fm}$ with norm of $\fm$ up to $180,000$ gives $J_\tau$ to
an accuracy of
approximately $12$ digits: 
\[
 J_\tau\simeq -4.828954817077 +i\cdot 4.534696532333.
\]
There is a point $P$ of infinite order in $E(K)$  having $x$-coordinate equal to $18883/2420\alpha - 16127/2420$ (this
was found using naive search algorithms); let $J$ denote its corresponding image in $\C/\Lambda_E$. Then the
equality
\[
 J_\tau=-2J \pmod{\Lambda_E}
\]
holds up to the computed accuracy, giving numerical evidence of the equality $P_\tau=-2P$ and, therefore, of the
rationality of $P_\tau$.

\end{example}

 \bibliographystyle{halpha}
 \bibliography{refs}

\end{document}

%% file: curve15.tex
\begin{tabular}{c|c|c}
$D$ & $h$ & $P^+$\\\hline\hline
$13$ & $1$ & $\left(-\sqrt{13} + 1, 2 \sqrt{13} - 4\right)$ \\
$28$ & $1$ & $\left(-15 \sqrt{7} + 43, 150 \sqrt{7} - 402\right)$ \\
$37$ & $1$ & $\left(-\frac{5}{9} \sqrt{37} + \frac{5}{9}, \frac{25}{27} \sqrt{37} - \frac{70}{27}\right)$ \\
$73$ & $1$ & $\left(-\frac{17}{32} \sqrt{73} + \frac{77}{32}, \frac{187}{128} \sqrt{73} - \frac{1199}{128}\right)$ \\
$88$ & $1$ & $\left(-\frac{17}{9}, \frac{14}{27} \sqrt{22} + \frac{4}{9}\right)$ \\
$97$ & $1$ & $\left(-\frac{25}{121} \sqrt{97} + \frac{123}{121}, \frac{375}{2662} \sqrt{97} - \frac{4749}{2662}\right)$ \\
$133$ & $1$ & $\left(\frac{103}{9}, \frac{92}{27} \sqrt{133} - \frac{56}{9}\right)$ \\
$172$ & $1$ & $\left(-\frac{1923}{1681}, \frac{11781}{68921} \sqrt{43} + \frac{121}{1681}\right)$ \\
$193$ & $1$ & $\left(\frac{1885}{288} \sqrt{193} + \frac{25885}{288}, \frac{292175}{3456} \sqrt{193} + \frac{4056815}{3456}\right)$ \\
\\\end{tabular}

%% file: curve21.tex
\begin{tabular}{c|c|c}
$D$ & $h$ & $P^+$\\\hline\hline
$8$ & $1$ & $\left(-9 \sqrt{2} + 11, 45 \sqrt{2} - 64\right)$ \\
$29$ & $1$ & $\left(-\frac{9}{25} \sqrt{29} + \frac{32}{25}, \frac{63}{125} \sqrt{29} - \frac{449}{125}\right)$ \\
$44$ & $1$ & $\left(-\frac{9}{49} \sqrt{11} - \frac{52}{49}, \frac{54}{343} \sqrt{11} + \frac{557}{343}\right)$ \\
$53$ & $1$ & $\left(-\frac{37}{169} \sqrt{53} + \frac{184}{169}, \frac{555}{2197} \sqrt{53} - \frac{5633}{2197}\right)$ \\
$92$ & $1$ & $\left(\frac{533}{46}, \frac{17325}{2116} \sqrt{23} - \frac{533}{92}\right)$ \\
$137$ & $1$ & $\left(-\frac{1959}{11449} \sqrt{137} + \frac{242}{11449}, \frac{295809}{2450086} \sqrt{137} - \frac{162481}{2450086}\right)$ \\
$149$ & $1$ & $\left(-\frac{261}{2809} \sqrt{149} + \frac{2468}{2809}, \frac{8091}{148877} \sqrt{149} - \frac{101789}{148877}\right)$ \\
$197$ & $1$ & $\left(-\frac{79135143}{209961032} \sqrt{197} + \frac{977125081}{209961032}, \frac{1439547386313}{1075630366936} \sqrt{197} - \frac{9297639417941}{537815183468}\right)$ \\&&\\\hline 
$D$ & $h$ & $h_D(x)$\\\hline\hline
$65$ & $2$  & $x^{2} + \left(\frac{61851}{6241} \sqrt{65} - \frac{491926}{6241}\right) x - \frac{403782}{6241} \sqrt{65} + \frac{3256777}{6241}$\\
\\\end{tabular}

%% file: curve33.tex
\begin{tabular}{c|c|c}
$D$ & $h$ & $P^+$\\\hline\hline
$13$ & $1$ & $\left(-\frac{1}{2} \sqrt{13} + \frac{3}{2}, \frac{1}{2} \sqrt{13} - \frac{7}{2}\right)$ \\
$28$ & $1$ & $\left(\frac{22}{7}, \frac{55}{49} \sqrt{7} - \frac{11}{7}\right)$ \\
$61$ & $1$ & $\left(-\frac{1}{2} \sqrt{61} + \frac{5}{2}, \sqrt{61} - 11\right)$ \\
$73$ & $1$ & $\left(-\frac{53339}{49928} \sqrt{73} + \frac{324687}{49928}, \frac{31203315}{7888624} \sqrt{73} - \frac{290996167}{7888624}\right)$ \\
$76$ & $1$ & $\left(-2, \sqrt{19} + 1\right)$ \\
$109$ & $1$ & $\left(-\frac{143}{2} \sqrt{109} + \frac{1485}{2}, \frac{5577}{2} \sqrt{109} - \frac{58223}{2}\right)$ \\
$172$ & $1$ & $\left(-\frac{51842}{21025}, \frac{2065147}{3048625} \sqrt{43} + \frac{25921}{21025}\right)$ \\
$184$ & $1$ & $\left(\frac{59488}{21609}, \frac{109252}{3176523} \sqrt{46} - \frac{29744}{21609}\right)$ \\
$193$ & $1$ & $\Big( \frac{94663533349261}{678412148664608} \sqrt{193} + \frac{1048806825770477}{678412148664608} ,$ \\
 &   & \hspace{1cm}$\frac{147778957920931299317}{12494688311813553741184} \sqrt{193} + \frac{30862934493092416035541}{12494688311813553741184}\Big)$ \\&&\\\hline 
$D$ & $h$ & $h_D(x)$\\\hline\hline
$40$ & $2$  & $x^{2} + \left(\frac{2849}{1681} \sqrt{10} - \frac{6347}{1681}\right) x - \frac{5082}{1681} \sqrt{10} + \frac{16819}{1681}$\\
$85$ & $2$  & $x^{2} + \left(\frac{119}{361} \sqrt{85} - \frac{1022}{361}\right) x - \frac{168}{361} \sqrt{85} + \frac{1549}{361}$\\
$145$ & $4$  & $x^{4} + \left(\frac{169016003453}{83168215321} \sqrt{145} - \frac{1621540207320}{83168215321}\right) x^{3}$\\
& & $+ \left(-\frac{1534717557538}{83168215321} \sqrt{145} + \frac{18972823294799}{83168215321}\right) x^{2} + \left(\frac{5533405190489}{83168215321} \sqrt{145} - \frac{66553066916820}{83168215321}\right) x$\\
& & $+ -\frac{6414913389456}{83168215321} \sqrt{145} + \frac{77248348177561}{83168215321}$\\
\\\end{tabular}

%% file: curve35.tex
\begin{tabular}{c|c|c}
$D$ & $h$ & $P^+$\\\hline\hline
$24$ & $1$ & $\left(\frac{12565}{19321} \sqrt{6} + \frac{31879}{19321}, \frac{4020800}{2685619} \sqrt{6} + \frac{12075417}{2685619}\right)$ \\
$41$ & $1$ & $\left(70 \sqrt{41} + 449, 2100 \sqrt{41} + 13443\right)$ \\
$61$ & $1$ & $\left(\frac{7444913385}{279945122} \sqrt{61} + \frac{58532610047}{279945122}, \frac{1805488279736505}{3312030738382} \sqrt{61} + \frac{14113780406002997}{3312030738382}\right)$ \\
$69$ & $1$ & $\left(\frac{63742245}{280513298} \sqrt{69} + \frac{526671623}{280513298}, \frac{1810980922695}{3322118988214} \sqrt{69} + \frac{16959961136217}{3322118988214}\right)$ \\
$76$ & $1$ & $\Big( -\frac{4398502037370}{1404725114521} \sqrt{19} + \frac{19299436937929}{1404725114521} ,$ \\
 &   & \hspace{1cm}$\frac{28304052715333334100}{1664895657706548931} \sqrt{19} - \frac{121810800584629037164}{1664895657706548931}\Big)$ \\
$89$ & $1$ & $\left(\frac{981}{100}, \frac{3563}{1000} \sqrt{89} - \frac{1}{2}\right)$ \\
$101$ & $1$ & $\left(\frac{7505}{10404}, \frac{310345}{1061208} \sqrt{101} - \frac{1}{2}\right)$ \\
$124$ & $1$ & $\left(-\frac{210}{1681} \sqrt{31} + \frac{12769}{1681}, \frac{35700}{68921} \sqrt{31} - \frac{1678197}{68921}\right)$ \\
$129$ & $1$ & $\Big( \frac{9526581863470}{129638878212649} \sqrt{129} + \frac{154639065911401}{129638878212649} ,$ \\
 &   & \hspace{1cm}$\frac{227155723851142702700}{1476056210913547737643} \sqrt{129} + \frac{5024283358306642389249}{1476056210913547737643}\Big)$ \\
$181$ & $1$ & $\left(-\frac{4166720}{31843449}, \frac{7889580565}{359385165414} \sqrt{181} - \frac{1}{2}\right)$ \\&&\\\hline 
$D$ & $h$ & $h_D(x)$\\\hline\hline
$104$ & $2$  & $x^{2} - \frac{87841}{9522} x + \frac{85397}{6348}$\\
$136$ & $2$  & $x^{2} + \left(\frac{132755895957027}{12703756878289} \sqrt{34} - \frac{805260717153160}{12703756878289}\right) x - \frac{4207164401474475}{12703756878289} \sqrt{34} + \frac{24540106232139359}{12703756878289}$\\
\\\end{tabular}

%% file: curve51.tex
\begin{tabular}{c|c|c}
$D$ & $h$ & $P^+$\\\hline\hline
$8$ & $1$ & $\left(\frac{1}{2}, \frac{1}{4} \sqrt{2} - \frac{1}{2}\right)$ \\
$53$ & $1$ & $\left(\frac{3}{2} \sqrt{53} + \frac{23}{2}, \frac{15}{2} \sqrt{53} + \frac{107}{2}\right)$ \\
$77$ & $1$ & $\left(\frac{5559}{55778} \sqrt{77} + \frac{78911}{55778}, \frac{2040153}{9314926} \sqrt{77} + \frac{17804737}{9314926}\right)$ \\
$89$ & $1$ & $\left(\frac{793511}{2401}, \frac{150079871}{235298} \sqrt{89} - \frac{1}{2}\right)$ \\
$101$ & $1$ & $\Big( -\frac{656788148124048}{108395925566683225} \sqrt{101} + \frac{108663526315570777}{108395925566683225} ,$ \\
 &   & \hspace{1cm}$\frac{432742605985104670344096}{35687772118459783422252125} \sqrt{101} - \frac{71551860216079551941383354}{35687772118459783422252125}\Big)$ \\
$137$ & $1$ & $\left(\frac{83}{81}, \frac{193}{1458} \sqrt{137} - \frac{1}{2}\right)$ \\
$149$ & $1$ & $\Big( -\frac{41662615293}{110013332450} \sqrt{149} + \frac{802189306199}{110013332450} ,$ \\
 &   & \hspace{1cm}$\frac{39791672228037249}{25801976926160750} \sqrt{149} - \frac{635290450369692907}{25801976926160750}\Big)$ \\
$152$ & $1$ & $\Big( -\frac{1915814571}{20670100441} \sqrt{38} + \frac{24731592007}{20670100441} ,$ \\
 &   & \hspace{1cm}$\frac{577303899566856}{2971761010503011} \sqrt{38} - \frac{7167395643538198}{2971761010503011}\Big)$ \\
$161$ & $1$ & $\left(\frac{62146167667}{49710362300}, \frac{8395974419456303}{53153799096521000} \sqrt{161} - \frac{1}{2}\right)$ \\
$188$ & $1$ & $\Big( \frac{3178296211866}{1135825194001} \sqrt{47} + \frac{22525829850817}{1135825194001} ,$ \\
 &   & \hspace{1cm}$\frac{21864116230891316004}{1210506836331759751} \sqrt{47} + \frac{148356498531472446055}{1210506836331759751}\Big)$ \\&&\\\hline 
$D$ & $h$ & $h_D(x)$\\\hline\hline
$104$ & $2$  & $x^{2} + \left(-\frac{992302702743}{1960400420449} \sqrt{26} - \frac{57132410901980}{1960400420449}\right) x - \frac{4968445297101}{1960400420449} \sqrt{26} + \frac{61480175149213}{1960400420449}$\\
$140$ & $2$  & $x^{2} - \frac{7073157}{13924} x + \frac{398237221}{55696}$\\
$185$ & $2$  & $x^{2} + \left(-\frac{908505900}{7532677681} \sqrt{185} - \frac{54207252962}{7532677681}\right) x - \frac{787814100}{7532677681} \sqrt{185} + \frac{45005684581}{7532677681}$\\
\\\end{tabular}

%% file: curve105.tex
\begin{tabular}{c|c|c}
$D$ & $h$ & $P^+$\\\hline\hline
$29$ & $1$ & $2 \cdot \left(\frac{5}{2} \sqrt{29} + \frac{29}{2}, \frac{25}{2} \sqrt{29} + \frac{133}{2}\right)$ \\
$44$ & $1$ & $\left(\frac{47}{36}, \frac{13}{54} \sqrt{11} - \frac{83}{72}\right)$ \\
$149$ & $1$ & $\left(\frac{41297}{48050} \sqrt{149} + \frac{554429}{48050}, \frac{28371039}{7447750} \sqrt{149} + \frac{340434623}{7447750}\right)$ \\
\\\end{tabular}